\documentclass[12pt,twoside]{article}

\usepackage{hyperref}

\usepackage[english]{babel}
\usepackage[all]{xy}
\usepackage{amssymb,amsmath,bbm,url,xr,a4wide,cleveref,color}
\usepackage[mathscr]{euscript}
\usepackage{mathrsfs}
\usepackage[margin=1.5in]{geometry}
\usepackage{amsmath}
\usepackage{amsthm}
\usepackage{amssymb}
\usepackage{amscd}
\usepackage{mathrsfs}
\usepackage{graphicx}
\usepackage{rotating, graphicx}
\usepackage{multirow}
\usepackage{mathtools}
\usepackage{mathrsfs}
\usepackage{changepage} 
\usepackage{paralist}

\usepackage{kantlipsum}
\usepackage{enumitem}
\usepackage{times}
\usepackage{subfig}
\usepackage{color}

\usepackage{enumitem}
\makeatletter
\def\namedlabel#1#2{\begingroup
    #2%
    \def\@currentlabel{#2}%
    \phantomsection\label{#1}\endgroup
}
\makeatother

\title{Birational invariance of the Chow-Witt group of zero-cycles}

\author{Niels Feld}

\date{}

\newtheorem{thm}[subsubsection]{Theorem}
\newtheorem{prop}[subsubsection]{Proposition}
\newtheorem{lm}[subsubsection]{Lemma}
\newtheorem{cor}[subsubsection]{Corollary}

\theoremstyle{remark}
\newtheorem{rem}[subsubsection]{Remark}

\newtheorem{ex}[subsubsection]{Example}

\theoremstyle{definition}
\newtheorem{df}[subsubsection]{Definition}

\newtheorem{paragr}[subsubsection]{}

\numberwithin{equation}{subsubsection}

\newcommand{\cO}{\mathcal O}

\newcommand{\cL}{\mathcal L} 

\newtheorem{innercustomthm}{Theorem}
\newenvironment{customthm}[1]
{\renewcommand\theinnercustomthm{#1}\innercustomthm}
{\endinnercustomthm}

\DeclareMathOperator{\Der}{D}

\newcommand{\WCor}{\widetilde{\mathbf{Cor}}}
\newcommand{\WRatCor}{{\widetilde{\mathbf{RatCor}}}}

\newcommand{\ab}{\mathscr Ab}

\newcommand{\orsch}{\textbf{orSchm}} 

\newcommand{\cotg}{\mathrm L}
\newcommand{\cotgb}{\tau} 
\newcommand{\detcotgb}{\omega} 

\DeclareMathOperator\GW{GW}

\usepackage{mathtools}

\DeclareMathOperator{\Hom}{Hom}

\DeclareMathOperator{\CH}{CH}
\newcommand{\CHW}{\widetilde{\mathrm{CH}}{}}

\newcommand{\FF} {\mathbf F}

\newcommand{\ZZ} {\mathbf Z}

\newcommand{\PP} {\mathbf P}

\DeclareMathOperator{\Spec}{Spec}

\newcommand{\Id}{\operatorname{Id}}

\newcommand\res{\operatorname{res}}
\newcommand\cores{\operatorname{cores}}
\newcommand\KMW{\underline{\operatorname{K}}^{MW}}

\newcommand\kMW{\operatorname{K}^{MW}}

\newcommand{\hM}{\mathcal{M}}
\newcommand{\cohM}{M}

\begin{document}

\maketitle

\begin{abstract}
	We prove that the Chow-Witt group of zero-cycles is a birational invariant of smooth proper schemes over a base field.
\end{abstract}

\tableofcontents

\section*{Introduction}

\par The notion of \emph{Milnor-Witt cycle modules} is introduced by the author in \cite{Feld1, Fel21} over a perfect field $k$ which, after slight changes, can be generalized to more general base schemes (see \cite{BHP22} for the case of a regular base scheme, and \cite{DegliseFeldJin22} for any base schemes).

\par The main example of a Milnor-Witt cycle module is given by the Milnor-Witt K-theory $\KMW$ (see \cite{BCDFO, Feld1, Fel21,Fel21b,Fel21c} for more details). 

\par To any MW-cycle module $M$ and any $k$-scheme $X$ equipped with a line bundle $l_X$, one can associated a Rost-Schmid complex $C_*(X,M,l_X)$ whose homology groups are called the called \textit{the Chow-Witt groups with coefficient in $M$}. In particular, if $M=\KMW$, one recovers the Chow-Witt groups $\CHW_*(X,l)$ (see \cite{Fasel18bis}) which are, in some sense, a \textit{quadratic refinement} of the classical Chow group $\CH_*(X)$.

\par A well-known consequence of intersection theory is that the Chow group $\CH_0(X)$ is a birational invariant. Indeed, this was proved in full generality in characteristic $0$, and for surfaces in characteristic $p>0$ in the fundamental work of Colliot-Thélène and Coray \cite[Prop. 6.3]{CTC79}. The case of an algebraically closed base field can be found in \cite[Example 16.1.11]{Fult}, but the proof works verbatim for an arbitrary field.
%
%

\par A natural question is whether or not the birational invariance holds true for the Chow-Witt group and, more generally, of the Chow-Witt groups with coefficients in a Milnor-Witt cycle module). It is easy to see that the Chow-Witt group in \textbf{cohomological} degree zero $\CHW^0$ is a birational invariant for smooth proper $k$-scheme (see \cite[Theorem 5.6]{Fel21}). In homological degree zero, the question is more complex.
%

\par 
Following ideas of Merkurjev \cite{KM13}, we prove that the Chow-Witt group of zero-cycles is a birational invariant for smooth proper schemes. More generally, we have:

\begin{customthm}{1}[see Theorem \ref{thm_birational_invariance}]
	The group $A_0(X,M)$ is a birational invariant of the smooth proper scheme $X$.
	\par In particular, the Chow-Witt group of zero-cycles $\CHW_0(X)$ is a birational invariant of the smooth proper scheme $X$.
\end{customthm}

\subsection*{Outline of the paper}

\par In Section \ref{section_constructions}, we explain how to build a special type of Milnor-Witt cycle module from a fix MW-module. Moreover, we define a cup product for oriented schemes.

\par In Section \ref{section_rational_corr}, we prove that the two previous constructions are compatible with each other in some sense. This allows us to define a composition of \textit{Milnor-Witt rational correspondences} and construct an associated pushforward map. Finally, we apply these results to prove that Chow-Witt group of zero-cycles is a birational invariant for smooth proper schemes.

\par In Appendix \ref{Section_recollections}, we recall the basic definitions of (cohomological) Milnor-Witt cycle modules along with the basic maps (pushforward, pullback, etc.). We then define the new class of \textit{oriented} schemes.

\subsection*{Acknowledgments}
The author deeply thanks Frédéric Déglise, Fangzhou Jin, Jean Fasel, Adrien Dubouloz, Marcello Bernardara, Marc Levine, Paul Arne \O stv\ae r, Bertrand Toën, Joana Cirici, Baptiste Calmès, Alexey Ananyevskiy.

\section*{Notations and conventions}

In this paper, schemes are noetherian and finite dimensional.
 We fix a base field\footnote{Many results of the present paper are in fact true over a more general base scheme.} $k$ and put $S=\Spec k$, and we fix a base ring of coefficients $R$. 
 If not stated otherwise, all schemes and morphisms of schemes are defined over $S$.
 A \emph{point} (resp. \emph{trait}, \emph{singular trait}) of $S$ will be a morphism of schemes $\Spec(k) \rightarrow S$ essentially of finite type and such

Conventions: a morphism $f:X \rightarrow S$ (sometime denoted by $X/S$) is:
\begin{itemize}
\item essentially of finite type
 if $f$ is the projective limit of a cofiltered system $(f_i)_{i \in I}$ of morphisms of finite type
 with affine and \'etale transition maps
 
\item lci if it is smoothable and a local complete intersection (\emph{i.e.} admits a global factorization $f=p \circ i$,
 $p$ smooth and $i$ a regular closed immersion);
\item essentially lci if it is a limit of lci morphisms with étale transition maps. 

\end{itemize}
Let $X/S$ be a scheme essentially of finite type. We put $X_{(p)}$ the set of $p$-dimensional points of $X$.

A point $x$ of $S$ is a map $x:\Spec(E) \rightarrow S$ essentially of finite type and such $E$ is a field.
 We also say that $E$ is a field over $S$.

Given a morphism of schemes $f:Y \rightarrow X$, we let $\cotg_f$ be its cotangent complex,
 an object of $\Der^b_{\mathrm{coh}}(Y)$,
 and when the latter is perfect (e.g. if $f$ is essentially lci), we let $\cotgb_f$ be its associated
 virtual vector bundle over $Y$, and by $\detcotgb_f$ the determinant of $\cotgb_f$.

 \par If not stated otherwise, $M$ is a (cohomological) Milnor-Witt cycle module, $X$ is an $S$-scheme, $l$ is a line bundle over $X$, and $p,q$ are integers.

%
%

%
%

\section{Main constructions}

\label{section_constructions}

\subsection{The relative perverse homology}

%
%

We follow \cite[§7]{Rost96}. In this section, we show that new Milnor-Witt cycle modules can be obtained from the Chow groups of the fibers of a morphism. 
\begin{paragr}
	Let $\rho:Q \to S$ be a morphism of finite type and let $M$ be a cohomological MW-cycle module over $Q$. Fix $l$ a line bundle over $Q$. For any field $F$ over $S$, denote by $Q_F=Q\times_B \Spec F$. We define an object function $A_p[\rho, M,l]$ on $\FF(S)$ by
	\begin{center}
		$A_p[\rho, M,l] = \bigoplus_{q \in \ZZ} A_p[\rho, M_q,l]$
	\end{center}
where
	\begin{center}
		$A_p[\rho, M_q,l](F)= A_p(Q_F,M_q, \detcotgb_{Q_F/Q}^{\vee}\otimes l)$.
	\end{center}
	Our aim is to show that $A_p[\rho, M,l]$ is in a natural way a Milnor-Witt cycle module over $S$.
\end{paragr}
%

\begin{paragr}
	All the properties of Milnor-Witt cycle modules except axiom $(C)$ hold already on complex level, i.e. for the groups $C_p(Q_F,M)$. Indeed, we denote by $\widehat{M}$ the object function on $\FF(B)$ defined by
	\begin{center}
		$\widehat{M}(F)=C_p(Q_F,M, \detcotgb_{Q_F/Q}^{\vee}\otimes l)= \bigoplus_{q\in \ZZ} C_p(Q_F,M_q, \detcotgb_{Q_F/Q}^{\vee})$.
	\end{center}
	We first describe its data as a Milnor-Witt cycle premodule. These will be denoted by $\widehat{\res}_{F/E}, \widehat{\cores}_{F/E}$, etc. in order to distinguish them from the data $\res_{F/E}, \cores_{F/E}$, etc. of $M$.
	\par For a morphism of fields $\phi:E\to F$, let $\overline{\phi}:Q_F \to Q_E$ be the induced map.
	\begin{enumerate}
		\item {\sc data D1} Define
		\begin{center}
			$\widehat{\res}_{F/E}:=
			\phi^!:
			C_p(Q_E,M_q, \detcotgb_{Q_E/Q}^{\vee})
			\to
			C_p(Q_F,M_q, \detcotgb_{Q_F/Q}^{\vee}).
			$
		\end{center}
		\item {\sc data D2} Assume $\phi$ finite. Define 
		\begin{center}
			$\widehat{\cores}_{F/E}:=
			\phi_*:
			C_p(Q_F,M_q, \cO_{Q_F})
			\to
			C_p(Q_E,M_q,\cO_{Q_E}).
			$
		\end{center}
		
		\item {\sc data D3}
		Simply take  the $\KMW$-module structure on $C_p(Q_F,M)$ described in \cite[§1.4 and §5.4]{DegliseFeldJin22}.
		\item 
		{\sc data D4} Denote by $\widetilde{Q}_v= Q \times_S \Spec \cO_v$, 
		the generic fiber $Q_F$ and the special fiber $Q_{\kappa(v)}$. Define
		\begin{center}
			$\widehat{\partial}_v:C_p(Q_F,M_q)
			\to
			C_{p-1}(Q_{\kappa(v)}, M_q)$
		\end{center}
		by $(\widehat{\partial_v})^x_y 
		=
		\partial^x_y$ with $\partial^x_y$ as in 
		\cite[§5.3.13]{DegliseFeldJin22} with respect to the scheme $\widetilde{Q_v}$.
	\end{enumerate}
\end{paragr}

\begin{thm}
	Keeping the previous notations, the object functor $\widehat{M}$ along with these data form a Milnor-Witt cycle premodule over $S$.
\end{thm}

\begin{proof}
	All the required properties follow from the rules and axioms for $M$ and from the functorial properties studied in \cite[§1.4 and §5.4]{DegliseFeldJin22}.
\end{proof}

\begin{paragr}
	Now, we want to relate the differentials for the MW-cycle premodule $\widehat{M}$ to the differentials for the MW-cycle module $M$.
	\par Let $X\to S$ be a scheme over $S$ and let $\widetilde{X}= Q \times_S X$. Then for $x,y$ in $X$, there is a map
	\begin{center}
		$\widehat{\partial^x_y}:
		\widehat{M}(x)
		\to
		\widehat{M}(y)$
	\end{center}
	as in \cite[§5.3.13]{DegliseFeldJin22}.
	By definition, this is a map
	
	\begin{center}
		$\widehat{\partial}^x_y:
		C_p(Q_{\kappa(x)},M)
		\to
		C_p(Q_{\kappa(y)}, M)$
	\end{center}
	between cycle groups with coefficients in $M$.
\end{paragr}

\begin{prop}
	\label{PropRost7.2}
	Let $\widetilde{x}, \widetilde{y}$ in $\widetilde{X}$ be points lying over $x,y \in X$, respectively, and assume that $\widetilde{x} \in (Q_{\kappa(x)})_{(q)}$ and $\widetilde{y}\in
	(Q_{\kappa(y)})_{(q)}$. Denote by $(\widehat{\partial^x_y})^{\widetilde{x}}_{\widetilde{y}}$ the component of $\widehat{\partial^x_y}$ with respect to $\widetilde{x}$ and $\widetilde{y}$. Then
	\begin{center}
		$(\widehat{\partial}^x_y)^{\widetilde{x}}_{\widetilde{y}}
		=
		\partial^{\widetilde{x}}_{\widetilde{y}}:
		M_q(\widetilde{x}, \detcotgb_{\widetilde{x}/S})
		\to
		M_{q-1}(\widetilde{y},\detcotgb_{\widetilde{y}/S})$.
	\end{center}
\end{prop}

\begin{proof}
	We may assume $\widetilde{y} \in 
	\overline{ \{ \widetilde{x}  \}
	}^{(1)}$, since otherwise both sides are trivial. The dimension inequality \cite[p. 85]{Matsumura80} shows then $y \in \overline{
		\{
		x
		\}
	}^{(1)}$.
	Let $v$ run through the valuations of $\kappa(x)$ with center $y$ in $X$. Moreover, let $w$ run through the valuations on $\kappa( \widetilde{x})$ with center $\widetilde{y}$ in $\widetilde{X}$. The restriction of any $w$ to $\kappa(x)$ is one of the valuations $v$. Let $\widetilde{w} \in Q_{\kappa(v)}$ be the center of $w$ in $\widetilde{X} \times_X \Spec \cO_v$. Now the claim follows from
	\[
	\begin{array}{rcl}
		(\widehat{\partial}^x_y)^{\widetilde{x}}_{\widetilde{y}}
		&
		=
		&
		(\sum_v 
		\widehat{\cores}_{\kappa(v)/\kappa(y)}
		\circ
		\widehat{\partial}_v
		)^{\widetilde{x}}_{\widetilde{y}}
		\\
		{}
		&
		=
		&
		\sum_v
		\sum_{w|v}
		(\widehat{\cores}_{\kappa(v)/\kappa(y)})^{\widetilde{w}}_{\widetilde{y}}
		\circ
		(
		\widehat{\partial}_v
		)^{\widetilde{x}}_{\widetilde{w}}
		\\
		{}
		&
		=
		&
		\sum_v
		\sum_{w|v}
		\cores_{\kappa(\widetilde{w}/\kappa(\widetilde{y})}
		\circ
		\cores_{\kappa(w)| \kappa(\widetilde{w})}
		\circ
		\partial_w
		\\
		{}
		&
		=
		&
		\sum_w
		\cores_{\kappa(w)/\kappa(\widetilde{y})}
		\circ
		\partial_w
		\\
		{}
		&
		=
		&
		\partial^{\widetilde{x}}_{\widetilde{y}}.
	\end{array}
	\]
\end{proof}
It follows from \cite[Proposition 1.4.6]{DegliseFeldJin22} that the data of the MW-cycle premodule $\widehat{M}$ commute with the differentials of the complex $C_*(Q_F,M)$. Passing to homology, we obtain data D1-D4 for the object function $A_q[\rho,M]$.

\begin{thm}
	Keeping the previous notations, the object function $A_p[\rho,M]$ together with these data is a Milnor-Witt cycle module over $S$.
\end{thm}

\begin{proof}
	The rules for the data of the MW-cycle premodule $A_p[\rho,M]$ are immediate from the rules for $\widehat{M}$. Moreover, axiom (FD) for $M$ and Proposition \ref{PropRost7.2} show that (FD) holds for $\widehat{M}$ and thus for $A_p[\rho,M]$. It remains to verify axiom (C).
	\par Consider the map
	\begin{center}
		$\xymatrix{
			C_p(Q_{\kappa(\xi)})
			\ar[r]^-{\delta}
			&
			C_{p-1}(Q_{\kappa(\xi)})
			\oplus
			\bigoplus_{x\in X^{(1)}}
			C_p(Q_{\kappa(x)})
			\oplus
			C_{p+1}(Q_{\kappa(x_0)})
			\ar[r]^-{\delta}
			&
			C_p(Q_{\kappa(x_0)})	
		}$
	\end{center}
	defined by $\delta^z_y= \partial^z_y$ with $\partial^z_y$ as in \cite[§5.3.13]{DegliseFeldJin22} with respect to the scheme $Q\times_B X$ (we have shortened the notation by omitting $M$).
	\par By Proposition \ref{PropRost7.2}, we are reduced to show $\delta \circ \delta = 0$. It suffices to check that
	$(\delta \circ \delta)^z_y=0$ for $z\in 
	(Q_{\kappa(\xi)})_{(q)}$ and 
	$y\in
	(Q_{\kappa(x_0)})_{(q)}$
	with $y\in \overline{\{ z\} }^{(2)}$ 
	(here $\overline{\{ z\} }$ is the closure of $z$ in $\widetilde{X}$). The dimension inequality \cite[p. 85]{Matsumura80} shows
	\begin{center}
		$Z^{(1)}
		\subset
		(Q_{\kappa(\xi)})_{(q-1)}
		\cup
		\bigcup_x
		(Q_{\kappa(x)})_{(q)}
		\cup
		(Q_{\kappa(x_0)})_{(q+1)}$
	\end{center}
	with $Z= \overline{\{ z\} }_{(y)}$. We are done by axiom (C) for $M$.
\end{proof}

\begin{df}
	Keeping the previous notations, the Milnor-Witt cycle module $A_p[\rho, M]$ is called the \textit{$p$-th relative perverse homology} of $M$ with respect to $\rho$.
\end{df}

\begin{rem}
	One should also obtain the results present in \cite[§8]{Rost96}. In particular, the MW-cycle module $A_q[\rho, M]$ could be used to give another proof of the homotopy invariance of the Rost-Schmid complex.
\end{rem}

\subsection{The cup product}

\begin{paragr}
	We follow ideas of Merkurjev \cite{Merkurjev03}.
	We work over a base field $k$. We fix $M$ a Milnor-Witt cycle module over $k$.
	
\end{paragr}
\begin{paragr}
		Let $M\times N \to P$ be a bilinear pairing of MW-cycle modules over $k$. Let $X,Y$ and $Z$ be smooth schemes over $k$ with $Y$ irreducible smooth and proper. Denote by $\Delta:Y \to Y\times Y$ the diagonal map. Let $q$ be an integer and $l_X$ (resp. $l_Y$, $l'_Y$, and $l'_Z$) a line bundle over $X$ (resp. $Y$, $Y$, and $Z$). Assume that $\detcotgb_{Y/k}\otimes l_Y \otimes l'_Y \simeq \cO_Y$.
	\par We have a $\cup$-product
	\begin{center}
		$\cup: A_r(X \times Y, M_s, l_X \otimes l_Y)
		\otimes
		A_p(Y\times Z,N_q,l'_Y\otimes l'_Z)
		\to
		A_{r+p-d_Y}(X\times Z,P_{s+q+d_Y},l_X \otimes l_Z)
		$
	\end{center}
	defined as the composition
	\begin{center}
		$\xymatrix{
			A_r(X \times Y, M_s,l_X \otimes l_Y)
			\otimes
			A_p(Y\times Z,N_q,l'_Y\otimes l'_Z)
			\ar[d]^-{\times}
			\\
			A_{r+p}(X\times Y \times Y \times Z,P_{s+q}, l_X \otimes l_Y \otimes l'_Y\otimes l'_Z)
			\ar[d]^{(\Id_{X} \otimes \Delta \otimes \Id_{Z})^*}
			\\
			A_{r+p-d_Y}(X\times Y\times Z, P_{s+q+d_Y},l_X \otimes  \detcotgb_{Y/k} \otimes l_Y \otimes l'_Y\otimes l'_Z)
			\ar[d]^-{\simeq}	
			\\
			A_{r+p-d_Y}(X\times Y\times Z, P_{s+q+d_Y},l_X \otimes l'_Z)
			\ar[d]^-{{\pi_{XZ}}_*}	
			\\
			A_{r+p-d_Y}(X\times Z,P_{s+q+d_Y},l_X \otimes l'_Z)
			}$
			
	\end{center}
	where $\times$ is the cross product (see \cite[Section 10]{Fel21}), $\Delta: Y \to Y \times Y$ is the diagonal embedding and $\pi_{XZ}:X\times Y \times Z \to X \times Z$ is the projection. The pushforward ${{p_{XZ}}_*}$ is well-defined because $Y$ is smooth and proper.

\end{paragr}

\begin{paragr}
	In particular, taking $N=M=P=\KMW$, $l_X=\detcotgb_{X/k}^{\vee}$, $l_Y=\cO_Y$, $l'_Y=\detcotgb_{Y/k}^{\vee}$, $l'_Z=\cO_Z$, $r=-s$ and $p=-q$, we have the product

	\begin{center}
		$\cup: \CHW_r(X \times Y,\omega_{X/S}^{\vee})
		\otimes
		\CHW_p(Y\times Z,\omega_{Y/S}^{\vee})
		\to
		\CHW_{r+p-d_Y}(X\times Z,\omega_{X/S}^{\vee})
		$
	\end{center}
	which could be taken as the composition law for the category of Milnor-Witt integral correspondences $\WCor$  with objects the smooth proper schemes over $k$ and morphisms
	\begin{center}
		$\Hom_{\WCor}(X,Y)
		=
		\bigoplus_i \CHW_{d_i}(X_i \times Y, \omega_{X/S}^{\vee})$,	
	\end{center}
	where $X_i$ are irreducible (connected) components of $X$ with $d_i= \dim X_i$.
\end{paragr}

\section{Milnor-Witt rational correspondences}
\label{section_rational_corr}

Let $X$ be a smooth and proper $k$-scheme and $l_X$ (resp. $l_Y$) a line bundle over $X$ (resp. $Y$). There is a canonical map of complexes
\begin{center}
	$\Theta_{M}:
	C_p(X\times Y, M_q,l_X \otimes l_Y)
	\to
	C_p(X, A_0[Y,M_{q},l_Y],l_X)$,
\end{center}
that takes an elements in $M(z,\detcotgb_z \otimes {l_X}_{|z}\otimes  {l_Y}_{|z})$ for $z\in (X \times Y)_{(p)}$ to zero if dimension of the projection $x$ of $z$ in $X$ is strictly less than $p$, and identically to itself otherwise. In the latter case, we consider $z$ as a point of dimension $0$ in $Y_x:=Y_{\kappa(x)}$ under the inclusion $Y_x \subset X \times Y$. Thus, $\Theta_{Y,M}$ ``ignores" points in $X \times Y$ that lose dimension being projected to $X$.
\par We study various compatibility properties of $\Theta_M$.
\subsection{Cross products}
	Let $M\times N \to P$ be a bilinear pairing of MW-cycle modules over $k$. For a smooth scheme $Y$ over $k$ and $l_Y$ a line bundle over $Y$, we can define a pairing
	\begin{center}
		$M
		\times A_0[Y,N, l_Y]
		\to A_0[Y,P,l_Y]$
	\end{center}
in an obvious way.

\begin{lm}
	\label{Lem_theta_and_product_commute}
	For $X,Y,Z$ smooth $k$-schemes, and $l_X$ (resp. $l_Y$, $l_Y$) a line bundle over $X$ (resp. $Y$, $Z$), the following diagram is commutative:
		\begin{center}
			\resizebox{13cm}{!}{
$\xymatrix{
C_p( X,M_q,l_X)
\otimes
C_r(Y\times Z,N_s,l_Y\otimes l_Z)
\ar[d]^{\Id \times \Theta_{N}}
\ar[r]^-{\times}
&
C_{p+r}(X\times Y \times Z, P_{q+s},l_X \otimes l_Y \otimes l_Z)
\ar[d]^{\Theta_{P}}
\\
C_p(X, M_q,l_X)
\otimes
C_r(YA_0[Z,N_s,l_Z],l_Y)
\ar[r]^-{\times}
&
C_{p+r}(X\times Y, A_0[Z,P_{q+s},l_Z],l_X\otimes l_Y).
}$	
}
	\end{center}
	
\end{lm}
\begin{proof}
	Let $x \in X_{(p)}$ and $\mu \in C_p(X,M_q,l_X)$. Consider the following commutative diagram

		\begin{center}
		$\xymatrix{
			C_r(Y\times Z, N_s, l_Y\otimes l_Z)
			\ar[r]^{\Theta_{N}}
			\ar[d]^{\pi'^*_x}
			&
			C_r(Y, A_0[Z,N_s,l_Z],l_Y)
			\ar[d]^{\pi^*_x}
			\\
			C_r((Y\times Z)_x, N_s, l_Y\otimes l_Z)
			\ar[r]^{\Theta_{N}}
			\ar[d]^-{m'_{\mu}}
			&
			C_{r}(Y_x, A_0[Z,N_s,l_Z],l_Y)
			\ar[d]^-{m_{\mu}}
			\\
			C_{p+r}( (Y\times Z)_x, P_{q+s}, l_X\otimes l_Y\otimes l_Z)
			\ar[r]^{\Theta_{P}}
			\ar[d]^-{i'_{x,*}}
			&
			C_{p+r}(Y_x, A_0[Z,P_{q+s},l_Z], l_X\otimes l_Y)
			\ar[d]^-{i_{x,*}}
			\\
			C_{p+r}(X\times Y \times Z, P_{q+s},l_X\otimes l_Y\otimes l_Z)
			\ar[r]^{\Theta_{P}}
			&
			C_{p+r}(X\times Y, A_0[Z,P_{q+s},l_Z],l_X\otimes l_Y)
			}$
			
		\end{center}
where $\pi_x:Y_x \to Y$ and $\pi'_x: (X\times Y)_x \to X\times Y$ are the natural projections, $m_{\mu}$ and $m'_{\mu}$ are the multiplications by $\mu$, and $i_x : Y_x \to X\times Y$ and $i'_x:(Y \times Z)_z \to X \times Y \times Z$ are the inclusions. By the definition of the cross product, the compositions in the two rows of the diagram are the multiplications by $\mu$.

\end{proof}

\begin{paragr}
	{\sc Pullback maps}
	Let $f:Z \to X$ be a regular closed embedding of smooth schemes of dimension $s$ and $l$ a line bundle over $X$. We denote by $N_{X/Z}$ the normal bundle over $Z$. For an smooth scheme $Y$, the closed embedding
	\begin{center}
		$f'=f\times \Id_Y : Z\times Y \to X\times Y$
	\end{center}
	is also regular and the normal bundle $N_{X\times Y/Z\times Y}$ is isomorphic to $N_{X/Z} \times Y$.
\end{paragr}

\begin{lm}
	\label{Lem2.2}
	The following diagram is commutative:
	\begin{center}
		$\xymatrix{
			A_p(X\times Y,M_q,l)
			\ar[r]^-{f'^*}
			\ar[d]^{\Theta_{M}}
			&
			A_{p+s}(Z\times Y,M_{q-s},l\otimes \detcotgb_f^{\vee})
			\ar[d]^{\Theta_{M}}
			\\
			A_p(X, A_0[Y,M_q],l)
			\ar[r]^-{f^*}
			&
			A_{p+s}(Z,A_0[Y,M_{q-s}],l\otimes \detcotgb_f^{\vee}).
		}$
		
	\end{center}	
\end{lm}

\begin{proof}
	Let ${\pi_X}:\textbf{G}_m \times X \to X$ and ${\pi'_X}:\textbf{G}_m \times X\times Y  \to X\times Y$ be the natural projections. The following diagram
	\begin{center}
		$\xymatrix{
			C_p(X\times Y, M_q,l)
			\ar[r]^-{{(\pi'_X})^{*}}
			\ar[d]^{\Theta_{M}}
			&
			C_{p+1}(\textbf{G}_m \times X\times Y, M_{q-1},l)
			\ar[d]^{\Theta_{M}}	
			\\
			C_p(X, A_0[Y,M_q],l)
			\ar[r]^-{\pi_X^*}
			&
			C_{q+1}(\textbf{G}_m \times X, A_0[Y,M_{q-1}],l)
		}$
		
	\end{center}
	is commutative.
	\par Let $t$ be the coordinate function on $\textbf{G}_m$. The map $\Theta_{M}$ commutes with the multiplication by $t$, i.e. the following diagram
	\begin{center}
		$\xymatrix{
			C_p( \textbf{G}_m \times  X\times Y,M_q,l)
			\ar[r]^{[t]}
			\ar[d]^{\Theta_{M}}
			&
			C_p(\textbf{G}_m \times X\times Y , M_{q+1},l)
			\ar[d]^{\Theta_{M}}
			\\
			C_p(\textbf{G}_m \times X, A_0[Y,M_q],l)
			\ar[r]^{[t]}
			&
			C_p(\textbf{G}_m \times X, A_0[Y,M_{q+1}],l)
		}$
		
	\end{center}
	is commutative.
	\par Let $D=D(X,Z)$ be the deformation space of the embedding $f$ (see e.g. \cite[§10]{Rost96}). There is a closed embedding $i:N_{X/Z}\to D$ with the open complement $j:\textbf{G}_m \times X \to D$. Then $D'=D\times Y$ is the deformation space $D(X\times Y, Z\times Y)$ with the closed embedding
	\begin{center}
		$i'=i\times \Id_Y :
		N_{X\times Y/Z\times Y} \to D'$
	\end{center}
	and the open complement $j'=j\times \Id_Y:\textbf{G}_m \times X\times Y \to D'$. 
	\par 
	The commutative diagram with exact rows
	
	\begin{center}
		$\xymatrix{
			\vdots
			\ar[d]
			&
			\vdots
			\ar[d]
			\\
			C_p(N_{X/Z}\times Y,M_q,l)
			\ar[r]^{\Theta_{M}}
			\ar[d]^-{i'_*}
			&
			C_p(N_{X/Z},A_0[Y,M_q],l)
			\ar[d]^-{i_*}
			\\
			C_p(D',M_q,l)
			\ar[r]^{\Theta_{M}}
			\ar[d]^-{j'^*}
			&
			C_p(D,A_0[Y,M_q],l)
			\ar[d]^-{j^*}
			\\	
			C_p(\textbf{G}_m \times X\times Y, M_q,l)
			\ar[r]^{\Theta_{M}}
			\ar[d]
			&
			C_p(\textbf{G}_m \times X, A_0[Y,M_q],l)
			\ar[d]
			\\
			\vdots
			&
			\vdots
		}$
		
	\end{center}
	%
	%
	%
	induces the commutative diagram
	
	\begin{center}
		$\xymatrix{
			C_p(\textbf{G}_m \times X\times Y ,M_q,l)
			\ar[r]^{\partial}
			\ar[d]^{\Theta_{M}}
			&
			C_{p-1}(N_{X/Z}\times Y,M_q,l)
			\ar[d]^{\Theta_{M}}
			\\
			C_p(\textbf{G}_m \times X, A_0[Y,M_q],l)
			\ar[r]^{\partial}
			&
			C_{p-1}(N_{X/Z}, A_0[Y,M_q],l).	
		}$
	\end{center}
	
	Finally, we also have the commutative diagram
	\begin{center}
		$\xymatrix{
			C_p(Z\times Y,M_q,l\otimes \detcotgb_f!!^{\vee})
			\ar[r]^-{\pi^*}
			\ar[d]^{\Theta_{M}}
			&
			C_{p+s}(N_{X/Z}\times Y,M_{q-s},l)
			\ar[d]^{\Theta_{M}}
			\\	
			C_{p}(Z, A_0[Y,M_q],l\otimes \detcotgb_f^{\vee})
			\ar[r]^-{\pi'^*}
			&
			C_{p+s}(N_{X/Z}, A_0[Y,M_{q-s}],l)
		}$
		
	\end{center}
	where $\pi : N_{X/Z} \to Z$ is the canonical projection and $s$ its relative dimension ($\pi$ is a quasi-isomorphism by homotopy invariance).
	By the definition of the pullback map (see \cite[Section 7]{Feld1}), the result follows from the composition of the previous commutative square.
\end{proof}

\begin{rem}
	The previous lemma could be stated at the level of complexes with the use of Rost's coordinations or by using the homotopy complex defined in \cite[§2.2]{DegliseFeldJin22}, but we do not need this generality.
\end{rem}

\begin{paragr}
	{\sc Pushforward maps}
	Let $f:X \to Z$ be a map of smooth schemes (over $k$). and $l$ a line bundle over $Z$. For an oriented smooth scheme $Y$, set
	\begin{center}
		$f'=f\times \Id_Y:
		X\times Y
		\to Z\times Y$.
	\end{center}
\end{paragr}

\begin{lm}
	\label{Lem_theta_and_pushforward_commute}
	The following diagram
	\begin{center}
		$\xymatrix{
			C_p(X\times Y,M_q,l)
			\ar[r]^{f'_*}
			\ar[d]^{\Theta_{M}}
			&
			C_p(Z\times Y,M_q,l)
			\ar[d]^{\Theta_{M}}
			\\
			C_p(X,A_0[Y,M_q],l)
			\ar[r]^{f_*}
			&
			C_p(Z, A_0[Y,M_q],l).
		}$		
		
	\end{center}
	
	is commutative.
\end{lm}

\begin{proof}
	Let $u\in (X\times Y)_{(p)}$, $a\in M(\kappa(u),\detcotgb_u \otimes l)$. Set $v=f'(u) \in Z\times Y$. If $\dim(v) < p$ then $(f'_*)_u(a)=0$. In this case, the dimension of the projection $y$ of $u$ in $Y$ is less than $p$ and hence $(\Theta_M)_u(a)=0$.
	\par Assume that $\dim(v)=p$. Then $\kappa(u)/\kappa(v)$ is a finite field extension and
	\begin{center}
		$b=(f'_*)_u(a)
		=
		\cores_{\kappa(u)/\kappa(v)}(a)
		\in M(\kappa(v), \detcotgb_v \otimes l)$.
	\end{center}
	If $\dim(y) < p$, then $(\Theta_M)_u(a) = 0$, and $\Theta_v(b)=0$. 
	\par 
	Assume that $\dim(y)=p$, then
	\begin{center}
		$(\Theta_M\circ
		f'_*)_u(a)
		=
		\cores_{\kappa(u)/\kappa(v)}(a)
		=
		b$
	\end{center}
	considered as an element of $A_0[Y,M_q](\kappa(z),\detcotgb_z \otimes l)=A_0(Y_z,M_q,l)$, where $z$ is the image of $v$ in $Z$. On the other hand, 
	\begin{center}
		$(f_* 
		\circ 
		\Theta_M)_u
		(a)
		=
		\phi_*(a)$,
	\end{center}
	where $\phi: Y_x \to Y_z$ is the natural map (where $x$ is the image of $u$ in $X$) and is considered as an element of $A_0[Y,M_q](\kappa(z), \detcotgb_z \otimes l)$. It remains to notice that
	\begin{center}
		$\phi_*(a)
		=
		\cores_{\kappa(u)/\kappa(v)}(a)=b$.
	\end{center}

\end{proof}

\subsection{Rational correspondences}
Let $Y$ and $Z$ be smooth schemes over $k$. Assume $Y$ irreducible and denote by $d_Y$ the dimension of $Y$.By Lemma \ref{Lem_theta_and_product_commute}, for the pairing $ M \times \KMW \to M$ and $``X=Y"$ we have the commutative diagram

\begin{center}
	$\xymatrix{
		A_{0}(Y, M_q)
		\otimes
		\CHW_{d_Y}(Y\times Z, \detcotgb_{Y/k}^{\vee})
		\ar[d]^{\Id\otimes \Theta_{\KMW}}
		\ar[r]^-{\times}
		&
		A_{d_Y}(Y \times Y \times Z, M_{-d_Y + q}, \detcotgb_{Y/k}^{\vee})
		\ar[d]^{\Theta_{M}}
		\\
		A_{0}(Y, M_q)
		\otimes
		A_{d_Y}(Y, A_0[Z, \KMW_{-d_Y}],\detcotgb_{Y/k}^{\vee})
		\ar[r]^-{\times}
		&
		A_{d_Y}(Y\times Y, A_0[Z,M_{-d_Y+q}],\detcotgb_{Y/k}^{\vee}).
	}$	
	
\end{center}
Let $\Delta:Y\to Y\times Y$ be the diagonal embedding and $\Delta'= \Delta \otimes \Id_Z$. By Lemma \ref{Lem2.2}, the following diagram
\begin{center}
	$\xymatrix{
		A_{d_Y}( Y\times Y \times Z,M_{-d_Y+q},\detcotgb_{Y/k}^{\vee})
		\ar[r]^-{\Delta'^*}
		\ar[d]^{\Theta_{M}}
		&
		A_0(Y\times Z,M_{q})
		\ar[d]^{\Theta_{M}}
		\\
		A_{d_Y}(Y\times Y, A_0[Z,M_{-d_Y+q}],\detcotgb_{Y/k}^{\vee})
		\ar[r]^-{\Delta^*}
		&
		A_0(Y,A_0[Z,M_{q}])	
	}$.
	
\end{center}

is commutative.

\par Finally, assume that the structure map $f:Y\to \Spec k$ is proper and denote by $f'=\Id_X\times f$. Lemma \ref{Lem_theta_and_pushforward_commute} implies that the following diagram

	\begin{center}
	$\xymatrix{
		A_0(Y\times Z,M_q)
		\ar[r]^-{f'_*}
		\ar[d]^{\Theta_{M}}
		&
		A_0(Z,M_q)
		\ar@{=}[d]
		\\
		A_0(Y,A_0[Z,M_q])
		\ar[r]^-{f_*}
		&
		A_0(\Spec k, A_0[Z,M_q]).
	}$	
		
\end{center}

is commutative.

\begin{prop}
	\label{Prop3.1}
	
	Let $Y$ and $Z$ be smooth schemes over $k$, $Y$  an irreducible smooth and proper, and $M$ an MW-cycle module over $k$. Then the pairing
	\begin{center}
		$\cup: A_{0}(Y, M_q)
		\otimes
		\CHW_{d_Y}(Y\times Z,\detcotgb_{Y/k}^{\vee})
		\to
		A_0(Z,M_q)	$
	\end{center}
 is trivial on all cycles in $\CHW_{d_Y}(Y\times Z,\detcotgb_{Y/k}^{\vee})$ that are not dominant over $Y$. In other words, the $\cup$-product factors through a natural pairing
	\begin{center}
		$\cup: A_{0}(Y, M_q)
		\otimes
		\CHW_{0}(Z_{\kappa(Y)},\detcotgb_{Y/k}^{\vee})
		\to
		A_0(Z,M_q)	$
	\end{center}
\end{prop}

\begin{proof}
	This follows from composing all three diagrams and taking into account that
	\begin{center}
		$A_{d_Y}(Y, A_0[Z,\KMW_{-d_Y}],\detcotgb_{Y/k}^{\vee})
		=
	\CHW_{0}(Z_{\kappa(Y)},\detcotgb_{Y/k}^{\vee}).$
	\end{center}
\end{proof}

\begin{paragr}
	Keeping the previous notations, for $Z$ irreducible smooth scheme over $k$, the diagram
	
	\begin{center}
		$\xymatrix{
	A_{d_X}(X\times Y, M_q,\detcotgb_{X/k}^{\vee})
	\otimes 
	\CHW_{d_Y}(Y\times Z,\detcotgb_{Y/k}^{\vee})
	\ar[r]^-{\cup}
	\ar[d]
	&
	A_{d_X}(X\times Z,M_q,\detcotgb_{X/k}^{\vee})
	\ar[d]
	\\
	A_0(Y_{\kappa(X)},M_q,\detcotgb_{X/k}^{\vee})
	\otimes
	\CHW_{d_Y}(Y\times Z,\detcotgb_{Y/k}^{\vee})
	\ar[r]^-{\cup}
	\ar[d]
	&
	A_0(Z_{\kappa(X)},M_q,\detcotgb_{X/k}^{\vee})
	\ar[d]
	\\
	A_0(Y_{\kappa(X)},M_q,\detcotgb_{X/k}^{\vee})
	\otimes
	\CHW_0(Z_{\kappa(Y)},\detcotgb_{Y/k}^{\vee})
	\ar[r]^-{\cup}
	&
	A_0(Z_{\kappa(X)},M_q,\detcotgb_{X/k}^{\vee})	
	}$
	
	\end{center}
	is commutative.
\end{paragr}

\begin{paragr}
	In particular, we have a well defined pairing
	\begin{center}
		$\cup:\CHW_0(Y_{\kappa(X)},\detcotgb_{X/k}^{\vee})
		\otimes
		\CHW_0(Z_{\kappa(Y)},\detcotgb_{Y/k}^{\vee})
		\to
		\CHW_0(Z_{\kappa(X)},\detcotgb_{X/k}^{\vee})$
	\end{center}
that can be taken for the composition law in the category of Milnor-Witt rational correspondences $\WRatCor(k)$ whose objects are the smooth proper schemes over $k$ and morphisms are given by
\begin{center}
	$\Hom_{\WRatCor(k)}(X,Y)
	=
	\bigoplus_i \CHW_0(Y_{\kappa(X_i)},\detcotgb_{X/k}^{\vee})$,
\end{center}
where $X_i$ are all irreducible (connected) components of $X$.
\par There is an obvious functor
\begin{center}
	$\Xi: \WCor(k) \to \WRatCor(k)$.
\end{center}
\end{paragr}

\begin{thm}
	For an MW-cycle module $M$, there exists a well-defined covariant functor
	\begin{center}
		$\WRatCor(k) \to \ab, X\mapsto A_0(X,M), a\mapsto -\cup a$.
	\end{center}
More precisely, the functor $\WCor(k) \to \WRatCor(k) $ factors through $\Xi$.
	
\end{thm}
\begin{proof}
	This follows from Proposition \ref{Prop3.1}.
\end{proof}

\begin{rem}
	Assuming one works with \textit{oriented} (see \ref{def_oriented_scheme}) smooth proper $k$-schemes, then there is also a contravariant functor given by $a \mapsto {}^ta \cup -$.
	We won't need this result.
\end{rem}

\begin{paragr}
	If $(\alpha : X \rightsquigarrow Y) \in \Hom_{\WRatCor(k)}(X,Y)$ is a MW-rational correspondence between two smooth proper $k$-schemes, we have a natural pushforward morphism
	\begin{center}
		$\alpha_*: A_0(X,M) \to A_0(Y,M)$.
	\end{center}
\end{paragr}

\begin{rem}
	If $\alpha$ et $\beta$ are two composable Milnor-Witt rational correspondences, then
	$$(\alpha \circ \beta)_*
	=
	\alpha_* \circ \beta_*.
	$$
\end{rem}

\begin{paragr}
	Let $f:X \dashrightarrow Y$ be a rational morphism of irreducible smooth $k$-schemes. It defines a rational point of $Y_{\kappa(X)}$ over $\kappa(X)$ and hence a morphism in $\Hom_{\WRatCor(k)}(X,Y)$ 
		that we denote by $[f]: X \rightsquigarrow Y$. In fact, the rational correspondence $[f]$ is the image of the class of the (transposed of the) graph of $f$ (as in \cite[Chapter 2, §4.3]{BCDFO}) under the natural map
		\begin{center}
			$\CHW_{d_X}(X\times Y, \detcotgb_{X/k})
			\to
			\CHW_0(Y_{\kappa(X)}, \detcotgb_{X/k})$.
		\end{center}
\end{paragr}

\begin{lm}
	\label{Lem_pushforward_rational_map}
	Let $\kappa/k$ be a finite type extension of fields. Let $f:X \dashrightarrow Y$ be a rational morphism of smooth proper $\kappa$-schemes and let $x\in X$ be a rational point such that $f(x)$ is defined. Denote by $[x] \in \CHW_0(X,\detcotgb_{\kappa/k})$ the $0$-cycle associated to $x$. Then 
	\begin{center}
		$[f]_*([x])= [ f(x)]$
	\end{center}
	in $\CHW_0(Y,\detcotgb_{\kappa/k})$.
\end{lm}
\begin{proof}
	Let $ \Gamma \subset X \times Y$ be the graph of $f$. 
	The preimage of $ \{x \} \times \Gamma $ under the morphism $  \Delta_X \otimes \Id_Y :X \times Y \to X \times X \times Y$ is the reduced scheme $\{x\} \times \{ f(x) \} $. 
	%
	Hence 
	$$
	[f]_*([x]) =
	[x] \cup [f]
	= \pi_* ( \Delta_X \otimes \Id_Y)^* ( [x] \times [\Gamma ]  )
	=
	\pi_*( [x]  \times  [f(x)]  )
	=
	[f(x)]
	$$ 
	where $\pi:X\times Y \to Y
	$
	is the projection.
\end{proof}

\begin{cor} \label{Cor_composition_corr_rat_maps}
	Let $f:X \dashrightarrow Y$ and $g:Y \dashrightarrow Z$ be composable rational morphisms of smooth proper schemes and let $h:X \dashrightarrow Z$ be the composition of $f$ and $g$. Then $[g]\circ [f] = [h]$ in $\Hom_{\WRatCor(k)} (X, Z)$.
\end{cor}
\begin{proof}
	Let $y$ be the rational point of $Y_{\kappa(X)}$ corresponding to $f$. By assumption, the rational morphism $g_{\kappa(X)} : Y_{\kappa(X)} \dashrightarrow Z_{\kappa(X)}$ is defined at $y$. By Lemma \ref{Lem_pushforward_rational_map} (with $``\kappa=\kappa(X)"$,$``X=Y_{\kappa(X)}"$, $``Y=Z_{\kappa(X)}$ and $``f=g_{\kappa(X)}"$) we see that the composition of correspondences $f$ and $g$ takes $[y]$ to $[g_{\kappa(X)}(y)] \in \CHW_0(Z_{\kappa(X)}, \detcotgb_{X/k}^{\vee})$. Note that the latter class corresponds to $h$.
\end{proof}

\begin{cor}
	\label{Cor_composition_rational_maps}
	For any two composable rational morphisms $f:X \dashrightarrow Y$ and $g: Y \dashrightarrow Z$ of smooth proper schemes, we have
	\begin{center}
		$[g\circ f]_*= [g]_* \circ [f]_*$.
	\end{center}

\end{cor}
\begin{proof}
	This is a consequence of Corollary \ref{Cor_composition_corr_rat_maps}.
\end{proof}

\begin{thm}
	\label{thm_birational_invariance}
	The group $A_0(X,M)$ is a birational invariant of the smooth proper scheme $X$.
	\par In particular, the Chow-Witt group of zero-cycles $\CHW_0(X)$ is a birational invariant of the smooth proper scheme $X$.
\end{thm}
\begin{proof}
	This is an immediate consequence of Corollary \ref{Cor_composition_rational_maps}.
\end{proof}

%

\begin{ex}
According to \cite[§5]{Fasel18bis}, we know that
$\CHW_0(\PP^n_k)=\GW(k)$ for any natural number $n$.
\par In particular, we recover the computations of $\CHW_0(Q_n)$ where $Q_n$ is an $n$-dimensional split quadric (see \cite[Corollary 9.5]{HXZ20}).
\end{ex}
\begin{ex}
	 If $M$ is $\operatorname{K}^M$ (the Milnor-Witt K-theory), then we recover the fact that the Chow group of zero-cycles $\CH_0(X)$ is a birational invariant of the smooth proper scheme $X$.
\end{ex}

\appendix
\section{Appendix}

\label{Section_recollections}

\subsection{Cohomological Milnor-Witt cycle modules}

\begin{df} 
	\label{def:cohMW}
	\begin{enumerate}
		
		\item
		If $S$ is a scheme, call an \textbf{$S$-field} the spectrum of a field essentially of finite type over $S$, and a \textbf{morphism of $S$-fields} an $S$-morphism between the underlying schemes. The collection of $S$-fields together with morphisms of $S$-fields defines a category which we denote by $\mathcal{F}_S$. We say that a morphism of $S$-fields is \textbf{finite} (resp. \textbf{separable}) if the underlying field extension is finite (resp. separable). 
		
		In what follows, we will denote for example $f:\Spec F\to\Spec E$ a morphism of $S$-fields, and $\phi:E\to F$ the underlying field extension.
		

		An \textbf{$S$-valuation} on an $S$-field $\operatorname{Spec}F$ is a discrete valuation $v$ on $F$ such that $\operatorname{Im}(\mathcal{O}(S)\to F)\subset\mathcal{O}_v$. We denote by $\kappa(v)$ the residue field, $\mathfrak{m}_v$ the valuation ideal and $N_v=\mathfrak{m}/\mathfrak{m}^2$.

		%

		\item
		Let $S$ be a scheme and let $R$ be a commutative ring with unit. An \textbf{$R$-linear cohomological Milnor-Witt cycle premodule} over $S$ is a functor from $\mathcal{F}_S$ to the category of $\ZZ$-graded $R$-modules
		\begin{align}
			\begin{split}
				M:(\mathcal{F}_S)^{op}&\to \operatorname{Mod}_R^{\ZZ}\\
				\operatorname{Spec}E&\mapsto M(E)
			\end{split}
		\end{align}
		for which we denote by $M_n(E)$ the $n$-the graded piece, together with the following functorialities and relations:
		
		\noindent\textbf{Functorialities:}
		\begin{description}
			\item [\namedlabel{itm:D1}{(D1)}] 
			For a morphism of $S$-fields $f:\operatorname{Spec}F\to \operatorname{Spec}E$ or (equivalently) $\phi:E \to F$, a map of degree $0$
			\begin{align}
				\label{eq:D1}
				f^*=\phi_*=\res_{F/E}:M(E)\to M(F);
			\end{align}
			
			\item [\namedlabel{itm:D3}{(D3)}] 
			For an $S$-field $\operatorname{Spec}E$ and an element $x\in \kMW_m(E)$, a map of degree $m$
			\begin{align}
				\gamma_x:M(E)\to M(E)
			\end{align}
			making $M(E)$ a left module over the lax monoidal functor $\kMW_?(E)$ (i.e. we have $\gamma_x\circ\gamma_y=\gamma_{x\cdot y}$ and $\gamma_1=\Id$).
		\end{description}
		
		The axiom~\ref{itm:D3} allows us to define, for every $S$-field $\operatorname{Spec}E$ and every $1$-dimensional $E$-vector space $\cL$, a graded $R$-module
		\begin{align}
			\label{eq:cohmorel}
			M(E,\cL):=M(E)\otimes_{R[E^\times]}R[\cL^\times]
		\end{align}
		where $R[\cL^\times]$ is the free $R$-module generated by the non-zero elements of $\cL$, and the group algebra $R[E^\times]$ acts on $M(E)$ via $u\mapsto\langle u\rangle$ thanks to~\ref{itm:D3}. 
		\begin{description}
			\item [\namedlabel{itm:D2}{(D2)}] 
			For a finite morphism of $S$-fields $f:\operatorname{Spec}F\to \operatorname{Spec}E$ or $\phi:E \to F$, a map of degree $0$ 
			\begin{align}
				\label{eq:D2}
				f_!=\phi^!= \cores_{F/E}:M(F,\detcotgb_{F/E})\to M(E);
			\end{align}
			
			\item [\namedlabel{itm:D4}{(D4)}] 
			For an $S$-field $\operatorname{Spec}E$ and an $S$-valuation $v$ on $E$, 
			a map of degree $-1$
			\begin{align}
				\label{eq:D4}
				\partial_v:M(E)\to M(\kappa(v), N_v^\vee).
			\end{align}
			
		\end{description}
		
		\noindent\textbf{Relations:} We refer to \cite[Definition 3.1]{Feld1} for the list of relations.
	\end{enumerate}

\end{df}

\begin{paragr} Fix $M$ a Milnor-Witt cycle premodule.
	If $X$ is any scheme, let $x,y$ be any points in $X$. We can define a map
	\begin{center}
		$\partial^x_y:M_q(\kappa(x),{\detcotgb}_{\kappa(x)/k}) \to M_{q-1}(\kappa(y),{\detcotgb}_{\kappa(y)/k})$
	\end{center}
	thanks to \ref{itm:D2} and \ref{itm:D4}.
\end{paragr}

\begin{df} (see \cite[Definition 4.2]{Feld1})
	
	A Milnor-Witt cycle module $M$ over $k$ is a Milnor-Witt cycle premodule $M$ which satisfies the following conditions \ref{itm:FD} and \ref{itm:C}.
	\begin{description}
		\item [\namedlabel{itm:FD}{(FD)}] {\sc Finite support of divisors.} Let $X$ be a normal scheme and $\rho$ be an element of $M(\xi_X,{}_X)$. Then $\partial_x(\rho)=0$ for all but finitely many $x\in X^{(1)}$.
		\item [\namedlabel{itm:C}{(C)}] {\sc Closedness.} Let $X$ be integral and local of dimension 2. Then
		\begin{center}
			$0=\displaystyle \sum_{x\in X^{(1)}} \partial^x_{x_0} \circ \partial^{\xi}_x: M(\kappa(\xi_X), \detcotgb_{\kappa(\xi_X)/k})\to M(\kappa(x_0),\detcotgb_{\kappa(x_0)/k})$
		\end{center}
		where $\xi$ is the generic point and $x_0$ the closed point of $X$.
	\end{description}
\end{df}

\begin{paragr}
	Let $M$ be a Milnor-Witt cycle module over $k$. We can form a (cohomological) Rost-Schmid cycle complex 
	$C_*(X,M,l)$ such that for any integer $p,q \in \ZZ$, and any line bundle $l$ over $X$:
	\begin{align}
		C_p(X,M_q,l):=\oplus_{X_{(p)}}M_{p+q}(\kappa(x),\detcotgb_{\kappa(x)/k} \otimes l_{|x}).
	\end{align}
	We denote by $A_i(X,M_q,l)$ is the homology of $C_*(X,M_q,l)$ in degree $i$. 
\end{paragr}

\begin{rem}
	Taking $M=\KMW$, we obtain
	$$
	A_i(X, M_{-i}, l) = \CHW_i (X, l )
	$$
	where the right-hand-side is known as the Chow-Witt group of $X$.
\end{rem}

\begin{paragr}
	\label{Cohomological basic maps}
	Fix $\cohM$ a Milnor-Witt cycle module and fix $X$ a $k$-scheme with a dimensional pinning.
	We recall the basic maps that one can define on the cohomological Rost-Schmid complex.
\end{paragr}

\begin{paragr}{\sc Pushforward}
	Let $f:Y\to X$ be a $k$-morphism of schemes. We have
	\begin{center}
		
		$f_*:C_p(Y,\cohM_q,l)\to {}C_{p}(X,\cohM_q, l)$
	\end{center}
	as follows. If $x=f(y)$ and if $\kappa(y)$ is finite over $\kappa(x)$, then $(f_*)^y_x=\cores_{\kappa(y)/\kappa(x)}$. Otherwise, $(f_*)^y_x=0$.
\end{paragr}

\begin{paragr}{\sc Pullback} \label{CohomologicalpullbackBasicMap}
	Let $f:Y\to X$ be an {\em essentially smooth} morphism of schemes of relative dimension $s$. Suppose $Y$ connected. Define
	\begin{center}
		$f^!:{}C_p(X,\cohM_q,l) \to {}C_{p+s}(Y,\cohM_{q-s},l\otimes \detcotgb_f^{\vee})$
	\end{center}
	as follows. If $f(y)=x$, then $(f^!)^x_y= \res_{\kappa(y)/\kappa(x)}$. Otherwise, $(f^!)^x_y=0$. If $Y$ is not connected, take the sum over each connected component.

\end{paragr}

\begin{paragr}{\sc Multiplication with units}
	Let $a_1,\dots, a_n$ be global units in $\mathcal{O}_X^*$. Define
	\begin{center}
		$[a_1,\dots, a_n]:
		{}C_p(X,\cohM_q,l) \to 
		{}C_p(X,\cohM_{q+n},l)$
	\end{center}
	as follows. Let $x$ be in $X_{(p)}$ and $\rho\in \hM(\kappa(x),*)$. We consider $[a_1(x),\dots, a_n(x)]$ as an element of ${\KMW (\kappa(x))}$.
	If $x=y$, then put $[a_1,\dots , a_n]^x_y(\rho)=[a_1(x),\dots , a_n(x)]\cdot \rho) $. Otherwise, put $[a_1,\dots , a_n]^x_y(\rho)=0$.

\end{paragr}

\begin{paragr}{\sc Multiplication with $\eta$}
	Define
	\begin{center}
		
		$\eta:
		{}C_p(X,\cohM_q,l)
		\to {}C_p(X,\cohM_{q-1},l)$
	\end{center}
	as follows. If $x=y$, 
	then $\eta^x_y(\rho)
	=\gamma_{\eta}(\rho)$. 
	Otherwise, $\eta^x_y(\rho)=0$.
	
\end{paragr}

\begin{paragr}{\sc Boundary maps} \label{CohomologicalBoundaryMaps}
	Let $X$ be a scheme of finite type over $k$, let $i:Z\to X$ be a closed immersion and let $j:U=X\setminus Z \to X$ be the inclusion of the open complement. We have a map
	\begin{center}

		$\partial=\partial^U_Z:
		{}C_{p}(U,\cohM_q,*) 
		\to {}C_{p-1}(Z,\cohM_q,*)$.
	\end{center}
	which is called the boundary map for the closed immersion $i:Z\to X$.
	
\end{paragr}

\begin{paragr}
	A pairing $N \times M \to P$ between MW-cycle modules is given by maps
	\begin{center}
		$M_p(E,l) \otimes N_q(E,l') 
		\to 
		P_{p+q}(E,l\otimes l')$
	\end{center}
	which are compatible with the data \ref{itm:D1},..., \ref{itm:D4} (see \cite[Definition 3.21]{Feld1} for more details).
\end{paragr}

\begin{paragr}{\sc Product} If $M\times N \to P$ is a pairing of Milnor-Witt cycle modules, then there is a product map
	\begin{center}
		$C_p(X,M_q,l)
		\times 
		C_{r}(Y, N_s, l')
		\to 
		C_{p+r}(X\times Y, P_{q+s}, l \otimes l')
		$
	\end{center}
	where $X,Y$ are smooth schemes over $k$ (see also \cite[§11]{Feld1}).
\end{paragr}

\begin{rem}
	The previous basic maps commute with the differentials of the Rost-Schmid complex and thus induce morphisms on the homology.
\end{rem}

\subsection{Oriented schemes}

\begin{paragr}
	\label{paragr_quadratic_iso}
	The notion of oriented real vector bundles was extended to the algebraic setting by Barges-Morel in \cite{BargeMorel}.
	We introduce a new category of oriented schemes. We refer to \cite[Appendix §6.1]{DDO} for similar results.
	
\end{paragr}

\begin{df}
	\label{def_oriented_scheme}
	Let $X/S$ be a scheme. An orientation of $X$ is an isomorphism $\sigma : \omega_{X/S} \to l_X^{\otimes 2}$, where $l_X$ is an invertible sheaf over $X$.
	\par An oriented $S$-scheme $(X,\sigma_X:\omega_{X/S} \to l_X^{\otimes 2})$ is the data of a scheme $X/S$ and an orientation $\sigma_X: \omega_{X/S} \to l_X^{\otimes 2}$.
	\par A morphism of oriented schemes $(Y, \sigma_Y:\omega_{Y/S} \to l_Y^{\otimes 2}) \to (X, \sigma_X: \omega_{X/S} \to l_X^{\otimes 2})$ is the data of an $S$-morphism $f:Y \to X$ along with an isomorphism of invertible sheaves
	$l_Y^{\otimes 2} \simeq f^{-1} l_X^{\otimes 2} \otimes \omega_f$ which makes the following diagram
	\begin{center}
		$\xymatrix{
			\omega_{Y/S}
			\ar[r]^-{\simeq}
			\ar[d]^-{\sigma_Y}
			&
			f^{-1} \omega_{X/S} \otimes \detcotgb_f
			\ar[d]^-{ \sigma_X \otimes \Id_{\detcotgb_f}}
			\\
			l_Y^{\otimes 2}
			\ar[r]^-{\simeq}
			&
			f^{-1} l_X^{\otimes 2} \otimes \detcotgb_f
		}
		$
		
	\end{center}
	commutative. 
	\par Denote by $\orsch$ the category of oriented schemes (along with morphisms of oriented schemes).
\end{df}

\begin{rem}
	Let $(X,\sigma_X:\omega_{X/S} \to l_X^{\otimes 2})$ be an oriented scheme. By abuse of notation, we omit the orientation and simply write $X$.
\end{rem}

\bibliographystyle{amsalpha}
\bibliography{MW}

\end{document}